\documentclass[11pt]{amsart}
\usepackage{amsfonts}
\usepackage{mathrsfs}
\usepackage{amscd}
\usepackage{amsmath}
\usepackage{amssymb}
\usepackage{latexsym}
\usepackage{lscape}
\usepackage{xypic}

\newtheorem{thm}{Theorem}[]

\newtheorem{lem}[thm]{Lemma}

\newtheorem{lem-def}[thm]{Lemma-Definition}

\theoremstyle{remark}

\newtheorem{rmk}{Remark}[section]
\theoremstyle{definition}

\numberwithin{equation}{section}

\input xy
\xyoption{all}

\newcommand{\quash}[1]{}  %%Anything in \quash is ignored
\newcommand{\nc}{\newcommand}
\nc{\on}{\operatorname}

\newcommand{\frakb}{{\mathfrak b}}

\newcommand{\frakg}{{\mathfrak g}}

\newcommand{\frakl}{{\mathfrak l}}

\newcommand{\fraks}{{\mathfrak s}}

\newcommand{\frakZ}{{\mathfrak Z}}
\newcommand{\bbA}{{\mathbb A}}

\newcommand{\bbC}{{\mathbb C}}

\newcommand{\bbG}{{\mathbb G}}

\newcommand{\bbM}{{\mathbb M}}

\newcommand{\bbP}{{\mathbb P}}

\newcommand{\bbW}{{\mathbb W}}

\newcommand{\calA}{{\mathcal A}}
\newcommand{\calB}{{\mathcal B}}

\newcommand{\calD}{{\mathcal D}}
\newcommand{\calE}{{\mathcal E}}
\newcommand{\calF}{{\mathcal F}}
\newcommand{\calG}{{\mathcal G}}

\newcommand{\calL}{{\mathcal L}}

\newcommand{\calO}{{\mathcal O}}

\newcommand{\calV}{{\mathcal V}}

\nc{\al}{{\alpha}} \nc{\be}{{\beta}}

\nc{\ve}{{\varepsilon}} \nc{\Ga}{{\Gamma}}
\newcommand{\la}{{\lambda}}
\nc{\La}{{\Lambda}}

\nc{\ad}{{\on{ad}}}

\nc{\Adm}{{\on{Adm}}} \nc{\aff}{{\on{aff}}}
\nc{\Aff}{{\mathbf{Aff}}}
\newcommand{\Aut}{{\on{Aut}}}
\nc{\Bun}{{\on{Bun}}}

\nc{\der}{{\on{der}}}

\nc{\diag}{{\on{diag}}}
\newcommand{\End}{{\on{End}}}
\nc{\Fl}{{\calF\ell}}

\newcommand{\Gr}{{\on{Gr}}}
\newcommand{\Hom}{{\on{Hom}}}
\nc{\IC}{{\on{IC}}}

\nc{\Id}{{\on{Id}}}

\nc{\Ind}{{\on{Ind}}}
\newcommand{\Lie}{{\on{Lie\ \!\!}}}

\nc{\res}{{\on{res}}}

\newcommand{\Spec}{{\on{Spec\ \!\!}}}

\nc{\tr}{{\on{tr}}}

\nc{\GSp}{{\on{GSp}}} \nc{\GU}{{\on{GU}}}

\nc{\SL}{{\on{SL}}} \nc{\SU}{{\on{SU}}} \nc{\SO}{{\on{SO}}}

\nc{\bFl}{{\overline{\Fl}}} \nc{\bU}{{\overline{U}}}
\nc{\wGr}{{\widetilde{\Gr}}} \nc{\wJG}{{\widetilde{\calL^+\calG}}}
\nc{\wLG}{{\widetilde{\calL\calG}}} \nc{\Op}{{\on{Op}}}
\nc{\crit}{{\on{crit}}} \nc{\Vac}{{\on{Vac}}}

\nc{\ppars}{(\!(s)\!)}

\title{Frenkel-Gross' irregular connection and Heinloth-Ng\^{o}-Yun's are the same}
\author{Xinwen Zhu}
\thanks{Partially supported by NSF grant DMS10-01280.}
\begin{document}
\maketitle

We show that the irregular connection on $\bbG_m$ constructed by
Frenkel-Gross (\cite{FG}) and the one constructed by
Heinloth-Ng\^{o}-Yun (\cite{HNY}) are the same, which confirms the
Conjecture 2.14 of \cite{HNY}.

The proof is simple, modulo the big machinery of quantization of
Hitchin's integrable systems as developed by Beilinson-Drinfeld
(\cite{BD}). The idea is as follows. Let $\calE$ be the irregular
connection on $\bbG_m$ as constructed by Frenkel-Gross. It admits a
natural oper form. We apply the machinery of Beilinson-Drinfeld to
produce an automorphic D-module on the corresponding moduli space of
$G$-bundles, with Hecke eigenvalue $\calE$. We show that this
automorphic D-module is equivariant with respect to the unipotent
group $I(1)/I(2)$ (see \cite{HNY} for the notation) against the
non-degenerate additive character $\Psi$. By the uniqueness of such
D-modules on the moduli space, one knows that the automorphic
D-module constructed using the Beilinson-Drinfeld machinery is the
same as the automorphic D-module explicitly constructed by
Heinloth-Ng\^{o}-Yun. Since the irregular connection on $\bbG_m$
constructed in \cite{HNY} is by definition the Hecke-eigenvalue of
this automorphic D-module, it is the same as $\calE$.

\section{Recollection of \cite{BD}}\label{recon} We begin with the review of the main results of
Beilinson-Drinfeld (\cite{BD}). We take the opportunity to describe
a slightly generalized (and therefore weaker) version of \cite{BD}
in order to deal with the level structures.

Let $G$ be a simple, simply-connected complex Lie group, with Lie
algebra $\frakg$ and the Langlands dual Lie algebra ${^L}\frakg$.
Let $X$ be a smooth projective algebraic curve over $\bbC$. For
every closed point $x\in X$, let $\calO_x$ be the completed local
ring of $X$ at $x$ and let $F_x$ be its fractional field. Let
$D_x=\Spec \calO_x$ and $D_x^\times=\Spec F_x$. In what follows, for
an affine (ind-)scheme $T$, we denote by $\on{Fun} T$ the
(pro)-algebra of regular functions on $T$.

Let $\calG$ be an integral model of $G$ over $X$, i.e. $\calG$ is a
(fiberwise) connected smooth affine group scheme over $X$ such that
$\calG_{\bbC(X)}=G_{\bbC(X)}$, where $\bbC(X)$ is the function field
of $X$. Let $\Bun_\calG$ be the moduli stack of $\calG$-torsors on
$X$. The canonical sheaf $\omega_{\Bun_\calG}$ is a line bundle on
$\Bun_\calG$. As $G$ is assumed to be simply-connected, we have
\begin{lem}
There is a unique line bundle $\omega_{\Bun_\calG}^{1/2}$ over $\Bun_\calG$, such that
$(\omega_{\Bun_\calG}^{1/2})^{\otimes 2}\simeq\omega_{\Bun_\calG}$.
\end{lem}

Now we assume that $\Bun_\calG$ is ``good" in the sense of
Beilinson-Drinfeld, i.e.
$$\dim T^*\Bun_\calG=2\dim\Bun_\calG.$$
In this case one can construct the D-module of the sheaf of
critically twisted (a.k.a. $\omega_{\Bun_\calG}^{1/2}$ twist)
differential operators on the smooth site $(\Bun_\calG)_{sm}$ of
$\Bun_\calG$, denoted by $\calD'$. Let $D'=(\underline{\End}
\calD')^{op}$ be the sheaf of endomorphisms of $\calD'$ as a twisted
D-module. Then $D'$ is a sheaf of associative algebra on
$(\Bun_\calG)_{sm}$ and $D'\simeq (D')^{op}$. For more details, we
refer to \cite[\S 1]{BD}.

Recall the definition of opers on a curve (cf. \cite[\S 3]{BD}). Let
$\on{Op}_{{^L}\frakg}(D^\times_x)$ be the ind-scheme of
${^L}\frakg$-opers on the punctured disc $D^\times_x$. Then there is
a natural ring homomorphism
\begin{equation}\label{hx}
h_x:\on{Fun}\on{Op}_{{^L}\frakg}(D^\times_x)\to
\Gamma(\Bun_\calG,D').
\end{equation}
Let us briefly recall its definition. Let $\Gr_{\calG,x}$ be the
affine Grassmannian, which is an ind-scheme classifying pairs
$(\calF,\beta)$, where $\calF$ is a $\calG$-torsor on $X$ and
$\beta$ is a trivialization of $\calF$ away from $x$. Then we have
$\Gr_{\calG,x}\simeq G(F_x)/K_x$, where $K_x=\calG(\calO_x)$. Let
$\calL_{\crit}$ be the pullback of the line bundle
$\omega_{\Bun_G}^{1/2}$ on $\Bun_\calG$ to $\Gr_{\calG,x}$, and let
$\delta_e$ be the delta $D$-module on $\Gr_{\calG,x}$ twisted by
$\calL_{\crit}$. Let
\[\Vac_x:=\Gamma(\Gr_{\calG,x},\delta_e)\]
be the vacuum $\hat{\frakg}_{\crit,x}$-module at the critical level.

\begin{rmk}The module $\Vac_x$ is not always isomorphic to $\Ind_{\Lie
K_x+\bbC\bf{1}}^{\hat{\frakg}_{\crit,x}}(\on{triv})$, due to the
twist by $\calL_{\crit}$. For example, if $K_x$ is an Iwahori
subgroup,
\[\Vac_x=\Ind_{\Lie
K_x+\bbC\bf{1}}^{\hat{\frakg}_{\crit,x}}(\bbC_{-\rho}),\] is the
Verma module of highest weight $-\rho$ ($-\rho$ is anti-dominant
w.r.t. the chosen $K_x$).
\end{rmk}

Let $\Bun_{\calG,x}$ be the scheme classifying pairs
$(\calF,\beta)$, where $\calF$ is a $\calG$-torsor on $X$ and
$\beta$ is a trivialization of $\calF$ on $D_x=\Spec \calO_x$. It
admits a $(\hat{\frakg}_{\crit,x},K_x)$ action, and
$\Bun_{\calG,x}/K_x\simeq \Bun_\calG$. Now applying the standard
localization construction to the Harish-Chandra module $\on{Vac}_x$
(cf. \cite[\S 1]{BD}) gives rise to
\[\on{Loc}(\on{Vac}_x)\simeq\calD'\]
as critically twisted $D$-modules on $\Bun_\calG$. Recall that the
center $\frakZ_x$ of the category of smooth
$\hat{\frakg}_{\on{crit},x}$-modules is isomorphic to
$\on{Fun}\on{Op}_{{^L}\frakg}(D^\times_x)$ by the Feigin-Frenkel
isomorphism (\cite[\S 3.2]{BD}, \cite{F}). The mapping $h_x$ then is
the composition
\[\on{Fun}\on{Op}_{{^L}\frakg}(D^\times_x)\simeq\frakZ_x\to\End(\on{Vac}_x)\to\End (\on{Loc}(\on{Vac}_x))\simeq\Gamma(\Bun_\calG,D').\]
If $\calG$ is unramified at $x$, then $h_x$ factors as
\[h_x: \on{Fun}\Op_{^{L}\frakg}(D_x^\times)\twoheadrightarrow\on{Fun}\Op_{^{L}\frakg}(D_x)\simeq\End(\Vac_x)\to\Gamma(\Bun_{\calG},D'),\]
where $\Op_{^{L}\frakg}(D_x)$ is the scheme (of infinite type) of
${^L}\frakg$-opers on $D_x$.

The mappings $h_x$ can be organized into a horizontal morphism $h$
of $\calD_X$-algebras over $X$ (we refer to \cite[\S 2.6]{BD} for
the generalities of $\calD_X$-algebras). Let us recall the
construction. By varying $x$ on $X$, the affine Grassmannian
$\Gr_{\calG,x}$ form an ind-scheme $\Gr_\calG$ formally smooth over
$X$. Let $\pi:\Gr_\calG\to X$ be the projection and
$e:X\to\Gr_\calG$ be the unital section given by the trivial
$\calG$-torsor. Let $\delta_e$ be the delta $D$-module along the
section $e$ twisted by $\calL_{\crit}$. Then we have a chiral
algebra
\[\calV ac_X:=\pi_!(\delta_e).\]
over $X$ whose fiber over $x$ is $\Vac_x$. \quash{As $\Gr_\calG\to
X$ is formally smooth,}
\begin{lem}The sheaf
$\calV ac_X$ is flat as an $\calO_X$-module.
\end{lem}

For any chiral algebra $\calA$ over a curve, one can associate the
algebra of its endomorphisms, denoted by $\calE nd(\calA)$. As
sheaves on $X$,
\[\calE nd(\calA)=\Hom_\calA(\calA,\calA),\]
where $\Hom$ is taken in the category of chiral $\calA$-modules.
Obviously, $\calE nd(\calA)$ is an algebra by composition. Less
obviously, there is a natural chiral algebra structure on $\calE
nd(\calA)\otimes\omega_X$ which is compatible with the algebra
structure. Therefore, $\calE nd(\calA)$ is a commutative
$\calD_X$-algebra. If $\calA$ is $\calO_X$-flat, there is a natural
injective mapping $\calE nd(\calA)_x\to \End(\calA_x)$ which is not
necessarily an isomorphism in general, where $\End(\calA_x)$ is the
endomorphism algebra $\calA_x$ as a chiral $\calA$-module. However,
this is an isomorphism if there is some open neighborhood $U$
containing $x$ such that $\calA|_U$ is constructed from a vertex
algebra. We refer to \cite{N} for details of the above discussion.

Let $U\subset X$ be an open subscheme such that $\calG|_U\simeq
G\times U$, then by the above generality, the Feigin-Frenkel
isomorphism gives rise to
\[\Spec\ \calE nd(\calV ac_U)\simeq \on{Op}_{{^L}\frakg}|_U,\]
where $\on{Op}_{{^L}\frakg}$ is the $\calD_X$-scheme over $X$, whose
fiber over $x\in X$ is the scheme of ${^L}\frakg$-opers on $D_x$.
Recall that for a commutative $\calD_U$-algebra $\calB$, we can take
the algebra of its horizontal sections $H_\nabla(U,\calB)$ (or
so-called conformal blocks) \cite[\S 2.6]{BD}, which is usually a
topological commutative algebra. For example,
\[\Spec
H_{\nabla}(U,\on{Op}_{{^L}\frakg})=\on{Op}_{{^L}\frakg}(U)\] is the
ind-scheme of ${^L}\frakg$-opers on $U$ (\cite[\S 3.3]{BD}). As
$H_{\nabla}(U,\calE nd(\calV ac_U))\to H_{\nabla}(X,\calE nd(\calV
ac_X))$ is surjective, we have a closed embedding
\[\Spec H_{\nabla}(X,\calE nd(\calV ac_X))\to\on{Op}_{{^L}\frakg}(U).\]
Let $\Op_{{^L}\frakg}(X)_\calG$ denote the image of this closed
embedding. This is a subscheme (rather than an ind-scheme) of
$\on{Op}_{{^L}\frakg}(U)$.

On the other hand, as argued in \cite[\S 2.8]{BD}, the mapping $h_x$
is of crystalline nature so that it induces a mapping of
$\calD_X$-algebras
\begin{equation}\label{hglob}
h: \calE nd(\calV ac_X)\to\Gamma(\Bun_\calG,D')\otimes\calO_X,
\end{equation}
which induces a mapping of horizontal sections
\begin{equation}\label{hhorizontal}
h_{\nabla}:H_{\nabla}(X,\calE nd(\calV
ac_X))\to\Gamma(\Bun_{\calG},D').
\end{equation}
Therefore, \eqref{hhorizontal} can be rewrite as a mapping
\begin{equation}\label{hhorizontal1}h_{\nabla}:\on{Fun}\on{Op}_{{^L}\frakg}(X)_{\calG}\to
\Gamma(\Bun_{\calG},D').\end{equation}

We recall the characterization $\on{Op}_{{^L}\frakg}(X)_{\calG}$.
\begin{lem}Let $X\setminus U=\{x_1,\ldots,x_n\}$.
Assume that the support of $\Vac_{x_i}$ (as an
$\frakZ_{x_i}$-module) is
$Z_{x_i}\subset\on{Op}_{{^L}\frakg}(D^\times_{x_i})$ (i.e.
$\on{Fun}(Z_{x_i})=\on{Im} (\on{Op}_{{^L}\frakg}(D_x^\times)\to
\End(\Vac_x))$). Then $$\on{Op}_{{^L}\frakg}(X)_{\calG}\simeq
\on{Op}_{{^L}\frakg}(U)\times_{\prod_i\on{Op}_{{^L}\frakg}(D^\times_{x_i})}\prod
Z_{x_i}.$$
\end{lem}

The mapping \eqref{hhorizontal1} is a quantization of a classical
Hitchin system. Namely, there is a natural filtration (\cite[\S
3.1]{BD}) on the algebra $\on{Fun}\on{Op}_{{^L}\frakg}(U)$ whose
associated graded is the algebra of functions on the classical
Hitchin space
$$\on{Hitch}(U)=\bigoplus_i\Gamma(U,\Omega^{d_i+1})$$
where $d_i$s are the exponent of $\frakg$ and $\Omega$ is the
canonical sheaf of $X$. On the other hand, there is a natural
filtration on $\Gamma(\Bun_\calG,D')$ coming from the order of the
differential operators. Then \eqref{hhorizontal1} is strictly
compatible with the filtration and the associated graded map gives
rise to the classical Hitchin map
\[h^{cl}: T^*\Bun_\calG\to \on{Hitch}(U).\]
\begin{rmk}The above map $h^{cl}$ factors through certain closed
subscheme $\on{Hitch}(X)_\calG\subset\on{Hitch}(U)$ whose algebra of
functions is the associated graded of
$\on{Fun}\on{Op}_{{^L}\frakg}(X)_\calG$.
\end{rmk}

The following theorem summarizes the main results of \cite{BD}.
\begin{thm}\label{BDmain}
Let $\chi\in
\on{Op}_{{^L}\frakg}(X)_{\calG}\subset\on{Op}_{{^L}\frakg}(U)$ be a
closed point, which gives rise to a ${^L}\frakg$-oper $\calE$ on
$U$. Let
$\varphi_\chi:\on{Fun}\on{Op}_{{^L}\frakg}(X)_{\calG}\to\bbC$ be the
corresponding homomorphism of $\bbC$-algebras. Then
\[\Aut_{\calE}:=(\calD'\otimes_{\on{Fun}\on{Op}_{{^L}\frakg}(X)_{\calG},\varphi_\chi}\bbC)\otimes\omega_{\Bun_\calG}^{-1/2}\]
is a Hecke-eigensheaf on $\Bun_\calG$ with respect to $\calE$
(regarded as a ${^L}G$-local system).
\end{thm}
\begin{rmk}The statement of the about theorem is weaker than the
main theorem in \cite{BD} in two aspects: (i) if $\calG$ is the
constant group scheme (the unramified case), then
$\on{Op}_{{^L}\frakg}(X)_{\calG}=\on{Op}_{{^L}\frakg}(X)$ is the
space of ${^L}\frakg$-opers on $X$. In this case, Beilinson and
Drinfeld proved that
\[\on{Fun}\on{Op}_{{^L}\frakg}(X)\simeq\Gamma(\Bun_G,D')\]
and therefore $\Aut_\calE$ is always non-zero in this case; (ii) in
the unramified case, the automorphic D-module $\Aut_\calE$ is
holonomic.

The proofs of both assertions are based on the fact that the
classical Hitchin map is a complete integrable system.  If the level
structure of $\calG$ is not deeper than the Iwahori level structure
(or even the pro-unipotent radical of the Iwahori group), then by
the same arguments, the above two assertions still holds. However,
it is not obvious from the construction that $\Aut_\calE$ is
non-zero for the general deeper level structure, although we do
conjecture that this is always the case. In addition, for arbitrary
$\calG$, the automorphic D-modules constructed as above will in
general not be holonomic. This is the reason that we need to use a
group scheme different from \cite{HNY} in what follows.
\end{rmk}

\section{} Now we specialize the group scheme $\calG$. Let $G$ be a
simple, simply-connected complex Lie group, of rank $\ell$. Let us
fix $B\subset G$ a Borel subgroup and $B^-$ an opposite Borel
subgroup. The unipotent radical of $B$ (resp. $B^-$) is denoted by
$U$ (resp. $U^-$). Following \cite{HNY}, we denote by $\calG(0,1)$
the group scheme on $\bbP^1$ obtained from the dilatation of
$G\times\bbP^1$ along $B^-\times\{0\}\subset G\times\{0\}$ and
$U\times\{\infty\}\subset G\times\{\infty\}$. Following \emph{loc.
cit.}, we denote $I(1)=\calG(0,1)(\calO_\infty)$.

Let $\calG(0,2)\to\calG(0,1)$ be the group scheme over $\bbP^1$ so
that they are isomorphic away from $\infty$ and
$\calG(0,2)(\calO_\infty)=I(2):=[I(1),I(1)]$. Then $I(1)/I(2)\simeq
\prod_{i=0}^{\ell}U_{\al_i}$, where $\al_i$ are simple affine roots,
and $U_{\al_i}$ are the corresponding root groups. Let us choose for
each $\al_i$ an isomorphism $\Psi_i:U_{\al_i}\simeq\bbG_a$. Then we
obtain a well-defined morphism
\[\Psi: I(1)\to I(1)/I(2)\simeq \prod_{i=0}^{\ell}U_{\al_i}\simeq\prod\bbG_a\stackrel{\on{sum}}{\to}\bbG_a.\]
Let $I_\Psi:=\ker\Psi\subset I(1)$.

As explained in \emph{loc. cit.}, there is an open substack of
$\Bun_{\calG(0,2)}$, which is isomorphic to $\bbG_a^{\ell+1}$. For
the application of Beilinson-Drinfeld's construction, it is
convenient to consider $\Bun_{\calG(0,\Psi)}$, where
$\calG(0,\Psi)\to\calG(0,1)$ is an isomorphism away from $\infty$
and $\calG(0,\Psi)(\calO_\infty)=I_\Psi\subset I(1)=\calG(0,1)$.
Then $\Bun_{\calG(0,2)}$ is a torsor over $\Bun_{\calG(0,\Psi)}$
under the group $I_\Psi/I(2)\cong\bbG_a^\ell$ and there is an open
substack of $\Bun_{\calG(0,\Psi)}$ isomorphic to $\bbG_a$.

\begin{lem}The stack $\Bun_{\calG(0,\Psi)}$ is good in the sense of \cite[\S 1.1.1]{BD}.
\end{lem}
\begin{proof}Since $\Bun_{\calG(0,\Psi)}$ is a principal bundle
over $\Bun_{\calG(0,1)}$ under the group $I(1)/I_\Psi\simeq \bbG_a$,
it is enough to show that $\Bun_{\calG(0,1)}$ is good. It is
well-known in this case $\Bun_{\calG(0,1)}$ has a stratification by
elements in the affine Weyl group of $G$ and the stratum
corresponding to $w$ has codimension $\ell(w)$ and the stabilizer
group has dimension $\ell(w)$. Therefore $\Bun_{\calG(0,1)}$ is
good.
\end{proof}
Let $S_w$ denote the preimage in $\Bun_{\calG(0,\Psi)}$ of the
stratum in $\Bun_{\calG(0,1)}$ corresponding to $w$. Then
$S_1\simeq\bbA^1$, and for a simple reflection $s$, $S_1\cup
S_s\simeq\bbP^1$. In particular, any regular function on
$\Bun_{\calG(0,\Psi)}$ is constant.

Let us describe $\on{Op}_{{^L}\frakg}(X)_{\calG(0,\Psi)}$ in this
case.

At $0\in\bbP^1$, $K_0=\calG(0,\Psi)(\calO_0)$ is the the Iwahori
subgroup $I^{op}$ of $G(F_0)$, which is $\on{ev}^{-1}(B^-)$ under
the evaluation map $\on{ev}: G(\calO_0)\to G$, and
\[\on{Vac}_0=\Ind_{\Lie I^{op} + \bbC \bf{1}}^{\hat{\frakg}_{\on{crit},x}}(\bbC_{-\rho}).\]
is just the Verma module $\bbM_{-\rho}$ of highest weight $-\rho$
($-\rho$ is anti-dominant w.r.t. $B^-$), and it is known
(\cite[Chap. 9]{F}) that
$\on{Fun}\on{Op}_{{^L}\frakg}(D^\times_0)\to\End (\bbM_{-\rho})$
induces an isomorphism
\[\on{Fun}\on{Op}_{{^L}\frakg}(D_0)_{\varpi(0)}\simeq\End(\bbM_{-\rho}),\]
where $\on{Op}_{{^L}\frakg}(D_0)_{\varpi(0)}$ is the scheme of
${^L}\frakg$ opers on $D_0$ with regular singularities and zero
residue. Let us describe this space in concrete terms. Let $f=\sum_i
X_{-\al_i}$ be the sum of root vectors $X_{-\al_i}$ corresponding
negative simple roots $-\al_i$ of ${^L}\frakg$. After choosing a
uniformizer $z$ of the disc $D_0$,
$\on{Op}_{{^L}\frakg}(D_0)_{\varpi(0)}$ is the space of operators of
the form
\[\partial_z+ \frac{f}{z}+{^L}\frakb[[z]]\]
up to ${^L}U(\calO_0)$-gauge equivalence.

At $\infty\in\bbP^1$, $K_\infty=\calG(0,\Psi)(\calO_\infty)=I_\Psi$.
Denote
\[\bbW_{univ}=\on{Vac}_\infty=\Ind_{\Lie I_{\Psi}+\bbC \bf{1}}^{\hat{\frakg}_{\on{crit},x}}(\on{triv}).\]
It is known (\cite[Lemma 5]{FF}) that
$\on{Fun}\on{Op}_{{^L}\frakg}(D^\times_\infty)\to\End (\bbW_{univ})$
factors as
\[\on{Fun}\on{Op}_{{^L}\frakg}(D^\times_\infty)\twoheadrightarrow\on{Fun}\on{Op}_{{^L}\frakg}(D_\infty)_{1/h}\hookrightarrow\End (\bbW_{univ}),\]
where $\on{Op}_{{^L}\frakg}(D_\infty)_{1/h}$ is the scheme of opers
with slopes $\leq 1/h$ (as ${^L}G$-local systems) and $h$ is the
Coxeter number of ${^L}G$. To give a concrete description of this
space, let us complete $f$ to an $\fraks\frakl_2$-triple
$\{e,\gamma,f\}$ with $e\in {^L}B$. Let ${^L}\frakg^e$ be the
centralizer of $e$ in ${^L}\frakg$, and decompose
${^L}\frakg^e=\oplus_{i=1}^{\ell}{^L}\frakg^e_i$ according to the
principal grading by $\gamma$. Let $d_i=\deg {^L}\frakg^e_i$. Then
after choosing a uniformizer $z$ on the disc $D_\infty$,
$\on{Op}_{{^L}\frakg}(D_\infty)_{1/h}$ is the space of operators of
the form
\[\partial_z+f+\sum_{i=1}^{\ell-1}+ z^{-d_i-1}({^L}\frakg^e_i)[[z]]+z^{-d_\ell-2}({^L}\frakg^e_\ell)[[z]]\]
up to ${^L}U(\calO_\infty)$-gauge equivalence.

Therefore, $\on{Op}_{{^L}\frakg}(X)_{\calG(0,\Psi)}$ is isomorphic
to
\[\on{Op}_{{^L}\frakg}(X)_{(0,\on{RS}),(\infty,1/h)}:=\on{Op}_{{^L}\frakg}(D_\infty)_{1/h}\times_{\on{Op}_{{^L}\frakg}(D_\infty^\times)}\on{Op}_{{^L}\frakg}(\bbG_m)\times_{\on{Op}_{{^L}\frakg}(D_0^\times)}\on{Op}_{{^L}\frakg}(D_0)_{\varpi(0)}.\]
As observed in \cite{FG},
$\on{Op}_{{^L}\frakg}(X)_{(0,\on{RS}),(\infty,1/h)}\simeq\bbA^1$.
Indeed, let $z$ be the global coordinate on
$\bbA^1=\bbP^1-\{\infty\}$. Then the space of such opers are of the
form
\[\nabla=\partial_z+ \frac{f}{z}+\lambda e_{\theta},\]
where $f$ is the sum of root vectors corresponding to negative
simple roots and $e_{\theta}$ is a root vector corresponding to the
highest root $\theta$.

According to \S \ref{recon}, there is a ring homomorphism
\[h_\nabla:\bbC[\la]\to\Gamma(\Bun_{\calG(0,\Psi)},D').\]
Let us describe this mapping more explicitly. Recall that there is
an action of $I(1)/I_\Psi\simeq \bbG_a$ on $\Bun_{\calG(0,\Psi)}$,
and therefore the action of $\bbG_a$ induces an algebra homomorphism
\[a: U(\Lie I(1)/I_\Psi)\to\Gamma(\Bun_{\calG(0,\Psi)},D').\]
\begin{lem}
We have $h_\nabla(\la)=a(\xi)$ for some non-zero element $\xi\in
\Lie I(1)/I_\Psi\simeq\bbC$.
\end{lem}
\begin{proof}Consider the associated graded
$h^{cl}:\on{gr}\bbC[\la]\to \Gamma(T^*\Bun_{\calG(0,\Psi)},\calO)$,
which is the classical Hitchin map. Recall that the filtration on
$\bbC[\la]$ comes from the existence of $\hbar$-opers, and therefore
the symbol of $\la$ is identified with a coordinate function on
\begin{eqnarray*}&&\on{Hitch}(X)_{\calG(0,\Psi)}\\ &\simeq&\bigoplus_{i=1}^{\ell-1}\Gamma(\bbP^1,\Omega^{d_i+1}((d_i)\cdot
0+
(d_i+1)\cdot\infty)\bigoplus\Gamma(\bbP^1,\Omega^{d_\ell+1}((d_\ell)\cdot
0+(d_\ell+2)\cdot\infty))\\ &\simeq& \bbA^1.\end{eqnarray*} On the
other hand, it is easy to identify the Hitchin map with the moment
map associated to the action of $I(1)/I_\Psi$ on
$\Bun_{\calG(0,\Psi)}$. Therefore, $h_\nabla(\la)=a(\xi)-c$ for some
constant $c$. Up to normalization, we can assume that
$d\Psi(\xi)=1$. We show that $c=0$. Indeed, consider the automorphic
D-module $\Aut=\calD'/\calD'\la$ on $\Bun_{\calG(0,\Psi)}$. It is
$I(1)/I_\Psi$-equivariant against $c\Psi$, with eigenvalue the local
system on $\bbG_m$ represented by the connection
$\partial_z+\frac{f}{z}$ by Theorem \ref{BDmain}, which is regular
singular. However, if $c\neq 0$, by \cite[Theorem 4(1)]{HNY}, the
eigenvalue for this $\Aut$ should be irregular at $\infty$.
Contradition.
\end{proof}

Finally, for any
$\chi\in\on{Op}_{{^L}\frakg}(X)_{(0,\on{RS}),(\infty,1/h)}$ given by
$\la=c$, $\Aut_{\calE}=\calD'/\calD'(\la-c)$ is a D-module on
$\Bun_{\calG(0,\Psi)}$, equivariant against $(I(1)/I_\Psi,c\Psi)$.
By the uniqueness of such D-modules on $\Bun_{\calG(0,\Psi)}$ (same
argument as in \cite[Lemma 2.3]{HNY}), this must be the same as the
automorphic D-module as constructed in \cite{HNY}. We are done.

\end{document}